\documentclass[12pt,reqno,a4paper]{amsart}
\usepackage{amsmath,amssymb,amsfonts,amscd,amsthm}
\usepackage[mathscr]{eucal}
\usepackage{comment}

\usepackage{diagrams}  
\usepackage{hyperref}  

\author{Theodore~Th. Voronov}
\address{School of Mathematics, University of Manchester, Manchester, M60 1QD, UK}
\email{theodore.voronov@manchester.ac.uk}
\address{Dept. of Quantum Field Theory, Tomsk State University, Tomsk, 634050, Russia}

\title[``Nonlinear  pullbacks'' and homotopy Poisson structures]{``Nonlinear  pullbacks'' of functions and $L_{\infty}$-morphisms for homotopy Poisson structures}

\newcommand{\new}[1]{{#1}} 

\newtheorem{theorem}{Theorem}
\newtheorem*{theore}{Theorem}

\newtheorem*{coro}{Corollary}
\newtheorem{lemde}[theorem]{Theorem-Definition}
\theoremstyle{definition}
\newtheorem{example}{Example}
\newtheorem{remark}{Remark}

\def\co{\colon\thinspace}

\renewcommand{\leq}{\leqslant}
\renewcommand{\geq}{\geqslant}

\DeclareMathOperator{\fun}{\mathit{C^{\infty}}}
\DeclareMathOperator{\funn}{\mathbf{C^{\,\infty}}}
\DeclareMathOperator{\pfunn}{\mathbf{\Pi\!C^{\,\infty}}}

\newcommand{\der}[2]{{\frac{\partial {#1}}{\partial {#2}}}}

\newcommand{\lder}[2]{{\partial {#1}/\partial {#2}}}
\newcommand{\dder}[3]{{\frac{\partial^2 {#1}}{\partial {#2}\partial {#3}}}}
\newcommand{\var}[2]{{\frac{\delta {#1}{ }}{\delta {#2}}}}

\newcommand{\Z}{{\mathbb Z_{2}}}

\newcommand{\p}{\partial}

\newcommand{\widebar}{\overline}

\renewcommand{\a}{\alpha}

\newcommand{\e}{\varepsilon}
\newcommand{\s}{\sigma}
\newcommand{\f}{{\varphi}}

\renewcommand{\o}{\omega}

\newcommand{\h}{\eta}

\newcommand{\F}{{\Phi}}

\renewcommand{\d}{\delta}

\renewcommand{\L}{{\Lambda}}

\newcommand{\itt}{{\tilde \imath}}

\newcommand{\at}{{\tilde a}}

\newcommand{\Ht}{{\tilde H}}

\newcommand{\Ft}{{\tilde F}}
\newcommand{\Gt}{{\tilde G}}

\newcommand{\ps}{{{\psi}}}

\newcommand{\X}{{\boldsymbol{X}}}

\newcommand{\Q}{{\boldsymbol{Q}}}

\DeclareMathOperator{\Vect}{\mathrm{Vect}}

\newcommand{\bL}{\mathbf{L}}

\newcommand{\lsch}{{[\![}}
\newcommand{\rsch}{{]\!]}}

\renewcommand{\S}{\Theta}

\date{6 (19) September 2016}
\begin{document}
\begin{abstract}
We introduce  mappings between  spaces of functions on   (super)manifolds that  generalize    pullbacks  with respect to   smooth maps  but  are, in general,  nonlinear (actually, formal). The construction is based on canonical relations and generating functions. (The underlying structure is a   formal category, which is a ``thickening'' of the usual category of supermanifolds; it   is  close to the category of symplectic micromanifolds and their micromorphisms considered recently by A.~Weinstein and A.~Cattaneo--B.~Dherin--A.~Weinstein.) There are two parallel settings, for even and odd functions. As an application, we show how such nonlinear pullbacks give  $L_{\infty}$-morphisms for algebras of functions on homotopy Schouten or homotopy Poisson manifolds.
\end{abstract}

\maketitle

\section*{Introduction}

In this paper we introduce and study a certain notion of a ``generalized pullback'' for functions on smooth manifolds or supermanifolds. It has two characteristic features: it is, in general, a \emph{nonlinear} mapping between the  spaces of functions (actually, it is given by a formal nonlinear differential operator), and it contains the usual pullback of functions as a particular case (when it is of course linear). Being nonlinear, such a ``generalized pullback'' cannot be an algebra homomorphism; however, it has the property that its derivative at each point is an algebra homomorphism.

The underlying ``generalized morphisms'' of (super)manifolds  are   certain canonical relations between the corresponding  cotangent bundles. Recall that a canonical relation (or a canonical correspondence) between symplectic manifolds is a Lagrangian submanifold in the direct product endowed with the difference of the symplectic forms. Such are the graphs of symplectomorphisms, hence canonical relations are usually seen as a generalization of symplectomorphisms.  Our viewpoint is   different (so thinking of symplectomorphisms is not useful  for   understanding) and can be explained as follows.  Let $\f\co M_1\to M_2$ be a usual smooth map.  It does not induce, in general, any map of the cotangent bundles. (The exception is the case of a diffeomorphism.) Nevertheless it gives a relation $R_{\f}\subset T^*M_1\times (-T^*M_2)$, which, as one can see, is a Lagrangian submanifold (here the minus sign means the negative of the symplectic form). This is our point of departure. As a generalization of maps $\f\co M_1\to M_2$ we consider canonical relations $\F\subset T^*M_1\times (-T^*M_2)$ of the type   ``closest to those of the form $R_{\f}$''. 
The latter requirement, as we will see, forces us to consider \emph{formal relations}, i.e., which live as submanifolds in the formal neighborhood of the zero section. All our constructions are therefore \emph{microformal} (based on power series in momentum variables).

Replacing ordinary maps by relations  and considering categories of relations is a very classical idea; in particular, the unifying role of canonical relations has been stressed by Weinstein~\cite{weinstein:symplectic-geometry1981, weinstein:symplcat82}.\footnote{Other manifestations of this idea include   additive relations in homological algebra~\cite{maclane:homology} and  the  Berezin--Neretin spinor representations of classical categories~\cite{berezin:second},~\cite{neretin:categories,neretin:categories-book}.}  The novelty  is not in  the use of relations as such, but in our choice of a particular class of relations and in their treatment for defining an analog of pullbacks. We see a relation $\F\subset T^*M_1\times (-T^*M_2)$   as a `morphism' not between $T^*M_1$ and $T^*M_2$, but   between $M_1$ and $M_2$. This is crucial and leads to the idea of a \textbf{pullback}
\begin{equation*}
    \F^*\co \funn(M_2)\to \funn(M_1)\,,
\end{equation*}
which we define as follows. Suppose $g\in \funn(M_2)$ is a function on $M_2$. The pullback $\F^*$ should map it to a function $f\in \funn(M_1)$. Consider the graph of the derivative $\Lambda_g\subset T^*M_2$. It is a Lagrangian submanifold. Then $f$ can be defined by the formula:
\begin{equation*}
    \Lambda_f:=\F\circ \Lambda_g \subset T^*M_1\,,
\end{equation*}
where $\circ$ denotes the composition of relations; one needs to show that $\F\circ \Lambda_g$ has the form $\Lambda_f$. More precisely, this defines the function $f$ up to a constant of integration, so our constructions involve a choice of constants. We assume that $\F\subset T^*M_1\times (-T^*M_2)$ can be specified by a generating function $S(x,q)$ depending on position variables on $M_1$ and momentum variables on $M_2$, and we treat $S$ as part of structure. Then the pullback $f=\F^*[g]$ introduced above in an abstract way is given explicitly by
\begin{equation*}
   \boxed{\quad f(x)=g(y) + S(x,q) - y^iq_i\,, \quad\vphantom{\int_0^1} }
\end{equation*}
where to eliminate the variables $q$ and $y$ one should use  the system of equations $q_i=\lder{g}{y^i}$ and $y^i=\lder{S}{q_i}$\,. It is solved by iterations, which gives $y$ as a function of $x$ depending in general on $g$ and its derivatives. Therefore the resulting $f(x)$ is expressed in $g(y)$ with all its derivatives perturbatively and nonlinearly, as a formal power series. (This is one of the many instances where formality enters the picture. Actually, the very assumption that we can use $S(x,q)$ implies a formal framework.) Remarkably, when $\F=R_{\f}$, the equations decouple, the nonlinearity disappears and the formula gives the ordinary pullback $\f^*g$.

We came to this construction motivated by the following task. Suppose a supermanifold $M$ is endowed with a homotopy analog of a Poisson structure. That means that\,---\,instead of a familiar bracket with two arguments\,---\,there is a whole sequence of brackets including binary, but also unary, ternary, etc., so that in particular the Jacobi identity for the binary bracket is satisfied up to an algebraic homotopy, where the ternary bracket is the homotopy and the unary bracket is the differential, and there are further identities involving the `higher homotopies'. Speaking more formally, there is an $L_{\infty}$-algebra  structure on functions on $M$ such that all the brackets are multiderivations with respect to the ordinary product of functions. (Actually, there are two different types of such structures on $M$, a `homotopy Poisson' and a `homotopy Schouten' structures, which differ by the parities of the brackets.) The problem is how to construct $L_{\infty}$-morphisms of such homotopy  structures. Let us concentrate on the homotopy Schouten case, for   concreteness. From general theory it is known that the most efficient description for $L_{\infty}$-algebras is that of $Q$-manifolds, i.e., supermanifolds endowed with homological vector fields (see, e.g.,~\cite{kontsevich:quant, kontsevich:quantlmp}). $L_{\infty}$-morphisms then correspond to maps of supermanifolds intertwining the corresponding homological  vector fields.  In the considered example,  homological vector fields   live on infinite-dimensional ``functional'' supermanifolds such as  $\funn(M)$  and one has to construct  mappings between them. This is a problem formulated in terms of infinite-dimensional geometry. On the other hand, a homotopy  Schouten structure has a  convenient finite-dimensional description in terms of differential-geometric objects on $M$ itself: it  is specified by an odd function $H$ on $T^*M$ satisfying $(H,H)=0$ for the canonical Poisson   bracket. Hence the question: is it possible to construct  $L_{\infty}$-morphisms in these terms?
Nonlinear pullbacks $\Phi^*$ that we introduce give the desired solution.

More precisely, we proved the following theorem. If  two odd Hamiltonians $H_1$ and $H_2$ specifying homotopy Schouten structures on $M_1$ and $M_2$ satisfy $p_1^*H_1=p_2^*H_2$ on $\F\subset T^*M_1\times (-T^*M_2)$ as above (in other words, if they are `$\Phi$-related',  for  $\Phi$   regarded as a `microformal morphism' from $M_1$ to $M_2$), then the pullback $\F^*\co \funn(M_2)\to \funn(M_1)$ is an $L_{\infty}$-morphism of the homotopy Schouten algebras. (Note that the   possible   nonlinearity of $\F^*$ is essential for obtaining nontrivial $L_{\infty}$-morphisms.)

Nonlinear pullbacks $\F^*$  as above are applicable to \textbf{even} functions. There is a parallel construction that works for \textbf{odd} functions and gives a different `formal thickening' of the category of smooth supermanifolds, with nonlinear pullbacks
\begin{equation*}
    \Psi^*\co \pfunn(M_2)\to \pfunn(M_1)\,
\end{equation*}
(here $\Pi$ is the parity reversion functor and `points' of $\pfunn(M)$ are odd functions on $M$).
It is based on   the anticotangent bundles $\Pi T^*M$ with the canonical odd symplectic structure.
There is a similar application to homotopy  Poisson  manifolds.

The formal categorical structure outlined only briefly in this paper, and  further applications to vector bundles and algebroids, are developed in~\cite{tv:microformal}. A certain `quantum' version  is constructed in~\cite{tv:oscil, tv:qumicro}.

In the course of this work, I greatly benefited from
discussions with H.~M.~Khudaverdian. Discussions
with K.~C.~H.~Mackenzie on various aspects of Lie algebroids and bialgebroids were a source of inspiration. I am very grateful to
A.~Weinstein and J.~Stasheff for their comments on early versions of this paper and to the anonymous referee for the remarks that helped to improve the exposition.

\section{Main construction}\label{sec.mainconstr}

Consider smooth manifolds or supermanifolds $M_1$ and $M_2$. We shall define a mapping of smooth functions on $M_2$ to smooth functions on $M_1$ which generalizes the pullback w.r.t. a smooth map $\f\co M_1\to M_2$. This mapping is, in general, nonlinear\,---\,actually, it will be defined as a formal power series\,---\,and has  an ordinary pullback as a particular case, in which is it is linear. (Moreover, as we shall see, ordinary pullbacks appear as linearizations or derivatives for our construction.)

\begin{remark} Constructions in this section are not specifically super, so the reader may  assume, initially, that we work with ordinary manifolds and ignore signs related with the supercase. Considering supermanifolds is necessary only for applications.
\end{remark}

To emphasize that we consider spaces of smooth functions as infinite-dimensional manifolds rather than vector spaces, we use boldface letters for the notation, e.g., $\funn(M)$ instead of $\fun(M)$.

\begin{remark}
This distinction is important for supermanifolds where $\funn(M)$ is itself an infinite-dimensional supermanifold  whose `points' are,  by definition, \emph{even} functions on $M$  possibly depending  on auxiliary odd parameters, unlike     elements of the $\Z$-graded vector space $\fun(M)$, which may be even or odd. Indeed, for consider unambiguously non-linear expressions involving $f(x)$, etc., the parity of $f$ should be fixed, since even and odd functions satisfy different commutativity constraints. (Similarly, odd functions on $M$ should be regarded as `points' of the infinite-dimensional supermanifold $\pfunn(M)$ corresponding to the $\Z$-graded vector space $\Pi\!\fun(M)$.)
\end{remark}

We shall use the language of canonical relations.
Consider the cotangent bundles $T^*M_1$ and $T^*M_2$. Denote local coordinates on $M_1$ and $M_2$ by $x^a$ and $y^i$, and let $p_a$ and $q_i$ stand for the corresponding conjugate momenta; so the canonical symplectic forms are $\o_1=dp_adx^a$ and $\o_2=dq_idy^i$, respectively.  We use local coordinates here because it is the most efficient language, though of course everything can be rephrased in a coordinate-free way.
It is well known that canonical relations between symplectic manifolds arise as graphs of  canonical transformations (symplectomorphisms) and may be seen as their generalizations. For   cotangent bundles, they also arise when one considers the effect of smooth maps of the bases. This is our starting point.

\begin{example} \label{ex.rphi}
Suppose $\f\co M_1\to M_2$ is a smooth map. It induces the following diagram for the cotangent bundles:
\begin{diagram}
T^*M_1 & \lTo^{{\;T^*\f}} & \f^*(T^*M_2) & \rTo^{\widebar{\f}} & T^*M_2\\
& \rdTo &    \dTo  & &  \dTo \\
& & M_1 & \rTo^{\f} & M_2
\end{diagram}
where the map $\widebar{\f}$ is fiberwise identical and at each point $x\in M_1$ the map $T^*\f(x)=(T\f)^*(x)$ is the adjoint of the tangent map. In coordinates,  $\f^*(y^i)=\f^i(x)$ and
\begin{align*}
\widebar{\f}&\co (x^a,q_i) \mapsto (y^i, q_i)  \  \ \text{where} \  \ y^i=\f^i(x)\,,\\
    T^*\f&\co (x^a,q_i)\mapsto (x^a,p_a)  \  \ \text{where} \  \ p_a=\der{\f^i}{x^a}(x)q_i\,.
\end{align*}
 We define a relation $R_{\f}\subset T^*M_1\times T^*M_2$ as follows:
\begin{equation}\label{eq.relrphi}
    R_{\f}=\left\{\,(x^a,p_a,y^i,q_i)\ \Bigl| \ p_a=\der{\f^i}{x^a}(x)q_i\,, \ y^i=\f^i(x)\,\Bigr.\right\}\,.
\end{equation}
In other words, $R_{\f}$ is the composition of the graphs of $T^*\f$ and $\widebar{\f}$. Note that the dimension of $R_{\f}$ is exactly half of the dimension of $T^*M_1\times T^*M_2$.  For the  pullbacks of the   symplectic forms $\o_1$  and $\o_2$ on $R_{\f}$ we obtain
\begin{align*}
    p_1^*\o_1&= p_1^*(dp_adx^a)=d \Bigl(\,\der{\f^i}{x^a}\,q_i\Bigr) dx^a= d \Bigl(\,\der{\f^i}{x^a}\,q_idx^a\Bigr) = d \Bigl(\,d\f^i q_i \Bigr)=dq_i d\f^i\,, \intertext{and}
    p_2^*\o_2&= p_2^*(dq_idy^i)= d q_i d\f^i\,.
\end{align*}
Therefore $p_1^*\o_1=p_2^*\o_2$ on $R_{\f}$, and so $R_{\f}$ is a Lagrangian submanifold in $T^*M_1\times T^*M_2$ considered with the symplectic structure $p_1^*\o_1-p_2^*\o_2$\,, i.e., a canonical relation.
\end{example}

This example serves as a model for our general construction. Denote by $T^*M_1\times (-T^*M_2)$ the symplectic manifold $T^*M_1\times T^*M_2$ considered with the symplectic form $\o_1-\o_2$ (where we have suppressed $p_1^*$ and $p_2^*$). Consider a canonical relation
\begin{equation*}
    \Phi\subset T^*M_1\times (-T^*M_2)
\end{equation*}
of the form generalizing~\eqref{eq.relrphi}\,,
\begin{equation}\label{eq.relr}
    \Phi =\left\{\,(x^a,p_a,y^i,q_i)\ \Bigl| \ p_a=\ps_a(x,q)\,, \ y^i=\f^i(x,q)\,\Bigr.\right\}\,.
\end{equation}
(We use the letter $\Phi$ is a reminiscent of a map of manifolds $\f$.) The condition that $\Phi$ is canonical implies that it can be described by a generating function $S=S(x,q)$  so that
\begin{equation} \label{eq.relr1.5}
    p_adx^a- (-1)^{\itt+1}y^idq_i =dS
\end{equation}
on $\Phi$. (Note that $p_adx^a- (-1)^{\itt+1}y^idq_i=  p_adx^a- q_idy^i + d(y^iq_i)$.) In other words, our relation has the form
\begin{equation}\label{eq.relr2}
    \Phi =\left\{\,(x^a,p_a,y^i,q_i)\ \Bigl| \ p_a=\der{S}{x^a}\,(x,q)\,, \ y^i=(-1)^{\itt}\der{S}{q_i}\,(x,q)\,\Bigr.\right\}\,,
\end{equation}
for a generating function $S=S(x,q)$.
(In particular,~\eqref{eq.relrphi} is recovered from~\eqref{eq.relr2} for $S=\f^i(x)q_i$\,.)

A good time to quote Arnold: ``Before turning to the apparatus of generating functions, we remark that it is unfortunately noninvariant\footnote{In the Russian original it sounds even stronger: ``depressingly noninvariant''.} and it uses, in an essential way, the coordinate structure in phase space''~\cite[p. 258]{arnold:mathmethodseng}.

\begin{remark} Recall that if $\L\subset N$ is a Lagrangian submanifold of a symplectic manifold, then in each local Darboux coordinate system on $N$ one can choose  half of the coordinates 
so that they are independent on $\L$ and  the remaining  canonically conjugated coordinates  are expressed as the partial derivatives  (up to signs) of a function of the coordinates from the first group. This function, by definition, is a \emph{generating function} for $\L$. By construction, it is a coordinate-dependent object. A particular type of such a function corresponds to   a choice of Darboux coordinates on $N$ taken as independent coordinates on the Lagrangian submanifold $\L$. Not all choices may be possible. In our situation, we  \emph{postulate} the possibility of taking $x^a,q_i$ as independent coordinates on $\Phi\subset T^*M_1\times (-T^*M_2)$ and hence the possibility of expressing our canonical relation in the form~\eqref{eq.relr}, \eqref{eq.relr2}.
\end{remark}

Now we shall define an analog of   pullback of  functions using a canonical relation $\Phi$ of the form~\eqref{eq.relr2} as a replacement of a smooth map $\f\co M_1\to M_2$. Consider an even function $g\in \funn(M_2)$.  The graph of its derivative is a Lagrangian submanifold $\L_g\subset T^*M_2$. In coordinates,
\begin{equation}\label{eq.grdg}
    \L_g =\left\{ (y^i,q_i)\ \Bigl| \ q_i=\der{g}{y^i}\,(y)\,\Bigr.\right\}\,.
\end{equation}
Take the  composition  of the relation $\Phi\subset T^*M_1\times (-T^*M_2)$ with the  submanifold $\L_g\subset T^*M_2$. It is a submanifold in $T^*M_1$,
\begin{equation*}
    \L:= \Phi\circ \L_g = p_1\left(\Phi\cap p_2^{-1}(\L_g)\right)\subset T^*M_1\,.
\end{equation*}
In coordinates,
\begin{equation}\label{eq.lam}
    \L =\left\{ (x^a,p_a)\ \Bigl| \ p_a=\der{S}{x^a}\,(x,q)\,, \ q_i=\der{g}{y^i}\,(y)\,, \ y^i=(-1)^{\itt}\der{S}{q_i}\,(x,q)\,\Bigr.\right\}\,.
\end{equation}

\begin{lemde}[preliminary version] \label{thm.one}
The submanifold $\L\subset T^*M_1$ is Lagrangian and is the graph of the derivative of a function $f\in \funn(M_1)$, $\L=\L_f$, which is given by
\begin{equation}\label{eq.f}
    f(x) =  g(y)+ S(x,q) - y^iq_i\,,
\end{equation}
where $q_i$ and $y^i$ in~\eqref{eq.f} are determined from the equations
\begin{align}\label{eq.q}
    q_i&=\der{g}{y^i}\,(y)\,, \\
    y^i&=(-1)^{\itt}\der{S}{q_i}\,\bigl(x,\der{g}{y}\,(y)\bigr)\,. \label{eq.y}
\end{align}
The function $f\in \funn(M_1)$ defined by Eq.~\eqref{eq.f}, together with Eqs.~\eqref{eq.q} and~\eqref{eq.y}, is called the \textbf{generalized pullback} of a function $g\in \funn(M_2)$ w.r.t. a canonical relation $\Phi\subset T^*M_1\times (-T^*M_2)$. Notation: $f=\Phi^*[g]$\,.
\end{lemde}
\begin{proof} The fact that $\L\subset T^*M_1$ is Lagrangian actually follows from general theory: the composition of Lagrangian relations, when it is well-defined, is Lagrangian. We shall check it directly together with establishing~\eqref{eq.f}.  On $\Phi$,   $p_adx^a- q_idy^i =dS- d(y^iq_i)$ (see~\eqref{eq.relr1.5}). By substituting $q_idy^i=dg$, we obtain that on $\L$, $p_adx^a =dg+ dS- d(y^iq_i)$, so $p_adx^a=df$ with $f$ given by~\eqref{eq.f}. In particular, $\L$   has the correct dimension, so it is indeed Lagrangian and $\L=\L_f$, as claimed.
Formulas~\eqref{eq.q} and \eqref{eq.y} follow from the definitions of $\Phi$ and $\L_g$\,. We shall   see that the construction of $f$ by formulas~\eqref{eq.f},\eqref{eq.q},\eqref{eq.y} does not depend on a choice of coordinates and   $f$ is globally defined.  This  requires some further clarifications that will be provided below after considering examples.
\end{proof}

Let us analyze formulas~\eqref{eq.f}, \eqref{eq.q} and \eqref{eq.y}  defining together the map $\Phi^*$.
In order to find the function $f=\Phi^*[g]\in\funn(M_1)$ from~\eqref{eq.f},  one has to solve equation~\eqref{eq.y} so to express $y$ as a function of $x$ and then substitute $y=y(x)$ into~\eqref{eq.q} and~\eqref{eq.f}.
The function $g\in\funn(M_2)$, the functional argument of the  mapping $\Phi^*$, enters the equation for $y$ through the derivative $\lder{g}{y}$. Ultimately, the function $g$ appears in~\eqref{eq.f} both explicitly as $g(y)$ and implicitly  through the variables $y$  and   $q$. To see the whole procedure better, we consider  examples (which will also lead us to promised clarifications).

\begin{example}[image of   zero] Let   $g=0\in \funn(M_2)$. We get $q_i=\lder{g}{y^i}=0$, so
\begin{equation*}
    f(x) =  S(x,0) - 0 +0=S_0(x)\,,
\end{equation*}
where we denoted $S_0(x):=S(x,0)$.  Therefore
\begin{equation}
    \Phi^*[0]=S_0\,.
\end{equation}
This is a fixed (even) function on $M_1$.
\end{example}

Let us write $S(x,q)$ in the form of a power expansion in $q_i$\,:
\begin{equation}\label{eq.sexpanded}
    S(x,q)=S_0(x)+ \f^i(x)q_i+\frac{1}{2}\,\new{S^{ij}}(x)q_jq_i+\frac{1}{3!}\,\new{S^{ijk}}(x)q_kq_jq_i + \ldots \
\end{equation}
(the reason for the choice of notation   $\f^i(x)$ will become clear shortly). 

Consider first a special case.

\begin{example}[ordinary pullback with a shift] \label{ex.pullbackshift}
Suppose
\begin{equation}\label{eq.sexpandedlinear}
    S(x,q)=S_0(x)+ \f^i(x)q_i\,
\end{equation}
(no higher order terms).
Then from~\eqref{eq.y},
\begin{equation*}
    y^i=\f^i(x)\,,
\end{equation*}
so the coefficients $\f^i$ define a smooth map $\f\co M_1\to M_2$. From~\eqref{eq.f} we obtain
\begin{equation*}
    f(x)=S_0(x)+ \f^i(x)q_i -y^iq_i+g(y)=S_0(x) + g(\f(x))\,.
\end{equation*}
We see that in this case, the map $\Phi^*$ is the combination of an ordinary  pullback w.r.t. a map $M_1\to M_2$ and a  `constant shift'  by a  function on $M_1$,
\begin{equation}\label{eq.linearappr}
    \Phi^*[g]= S_0+ \f^*g\,.
\end{equation}
\end{example}
In contrast with~\eqref{eq.sexpandedlinear},  it is not possible in a meaningful way to consider the generating function $S$ with a finite number of higher order terms  restricting  their order by some  $r>1$ (as we shall see later). 
So we have to deal with the whole expansion~\eqref{eq.sexpanded}. For that, we  consider $\Phi^*$ perturbatively near $g\equiv 0$.

\begin{example}[linear approximation for $\Phi^*$] \label{ex.linear}
Consider $g(y)=\e h(y)$ where $\e^2=0$. From~\eqref{eq.f} and \eqref{eq.q}, we see that we need to determine $y$ from~\eqref{eq.y} only in the zero order approximation. Also, there is no input  from the terms of order $\geq 2$ in ~\eqref{eq.sexpanded}. Hence $y^i=\f^i(x)$ and $f(x)$ is obtained similarly to Example~\ref{ex.pullbackshift}:
\begin{equation*}
    f(x)=S_0(x)+ \f^i(x)q_i -y^iq_i+\e h(y)=S_0(x) + \e h(\f(x))\,.
\end{equation*}
In other words,
\begin{equation}
    \Phi^*[\e h]= S_0 + \e \, \f^*h \mod \e^2\,.
\end{equation}
\end{example}

\begin{example}[quadratic approximation for $\Phi^*$] \label{ex.quadraticphi}To make one more step, let $g(y)=\e h(y)$ where $\e^3=0$. Now we see that we need $y$ in the linear approximation. Writing $y_{\e}=y_0+\e y_1 \mod \e^2$, we obtain from~\eqref{eq.y} and \eqref{eq.sexpanded}
\begin{equation*}
    y^i_{\e}= \f^i(x) +\e\,\new{S^{ij}}(x) \der{h}{y^j}\,(y_{0}) = \f^i(x) +\e\,\new{S^{ij}}(x) \der{h}{y^j}\,(\f(x)) \mod \e^2\,.
\end{equation*}
To find $f(x)$, we substitute into~\eqref{eq.f} and simplify:
\begin{multline*}
    f(x)= S(x,q)-y^iq_i +\e h(y)= \\
    S\left(x, \e \der{h}{y}\,(y_0 +\e y_1)\right) - \e (y^i_0+\e y^i_1)\der{h}{y^i}\,(y_0+\e y_1) +
    \e h(y_0+\e y_1)= \\
    S_0(x) + \e\, \f^i(x) \der{h}{y^i}\,(y_0 +\e y_1) + 
    \e^2 \,\frac{1}{2}\,\new{S^{ij}}(x)\der{h}{y^j}\,(y_0)\der{h}{y^i}\,(y_0)
    -\\
    \e (y^i_0+\e y^i_1)\der{h}{y^i}\,(y_0+\e y_1) +\e h(y_0+\e y_1)=\\
    S_0(x) +     \e^2 \,\frac{1}{2}\,\new{S^{ij}}(x)\der{h}{y^j}\,(y_0)\der{h}{y^i}\,(y_0)
    - \e^2\, y^i_1\der{h}{y^i}\,(y_0+\e y_1) +\e h(y_0+\e y_1)=
    \end{multline*}
    \begin{multline*}
    S_0(x) +  \e^2 \,\frac{1}{2}\,\new{S^{ij}}(x)\der{h}{y^j}\,(y_0)\der{h}{y^i}\,(y_0)
    - \e^2\, y^i_1\der{h}{y^i}\,(y_0) +\e\, h(y_0) +\e^2\, y_1^i\der{h}{y^i}\,(y_0) =\\
    S_0(x) + \e\, h(\f(x)) +  \e^2 \,\frac{1}{2}\,\new{S^{ij}}(x)\der{h}{y^j}\,(\f(x))\der{h}{y^i}\,(\f(x)) \mod \e^3\,.
\end{multline*}
Thus
\begin{equation}
    \Phi^*[\e h]= S_0 + \e \, \f^*h + \e^2 \,\frac{1}{2}\,\new{S^{ij}}\;\f^*\p_jh\,\f^*\p_ih   \mod \e^3\,.
\end{equation}
\end{example}

Generalizing from these examples, we  claim that in general the nonlinear transformation $\Phi^*$ \emph{exists at least at the formal level} as a perturbation series around an ordinary pullback  plus a  shift (a transformation of the form~\eqref{eq.linearappr})\,:
\begin{equation}\label{eq.phiexpanded}
    \Phi^*[g](x)= S_0(x) +  \f^*g(x) + \sum_{r\geq 2} \Phi_r\bigl(x,\f^*\p g(x),\f^*\p^2g(x),\ldots \bigr)\,,
\end{equation}
where each term  $\Phi_r$ is a homogeneous differential polynomial  in $g$  of order $\leq r$. (Also, $\Phi_r$ depends on derivatives  of degrees $\leq r$ in $g$.) Here the `shift'  $S_0(x)$ is given by the zero order term of a generating function $S(x,q)$ and the ordinary pullback $\f^*$ is with respect to a map $\f\co M_1\to M_2$, which is given the    first order terms  of the  function $S(x,q)$.

That this is indeed so  can be proved as follows. For a given $g$, equation~\eqref{eq.y} defines the variables $y^i$ as functions of $x^a$, i.e., it defines a smooth map 
\begin{equation}\label{eq.phig}
    \f_g\co M_1\to M_2\, 
\end{equation}
depending on $g\in \funn(M_2)$, which is a formal perturbation of  a given map $\f=\f_0\co M_1\to M_2$.
Solution of~\eqref{eq.y}  may be obtained by    an iterative procedure starting from $y^i=\f^i(x)=(-1)^{\itt}\lder{S}{q_i}(x,0)$. In other words, we write
\begin{equation*}
    \f_g=\f_0+\ldots \ \,,
\end{equation*}
where $\f_0=\f$, so the map $\f_g$ depending on $g$ is expressed as a perturbation series around the map $\f$ defined by the canonical relation $\Phi$ alone. Introduce parameter $\e$ and consider $\e g$ instead of $g$.  For the $N$th iterative step we write
\begin{equation*}
    y_{(N)}=y_0+\e\,y_1+\e^2y_2 +\ldots + \e^Ny_N \quad \mod \e^{N+1}\,,
\end{equation*}
where $y_0=\f(x)$. The term $y_N$ is defined via $y_0, y_1, \ldots, y_{N-1}$ from the equation
\begin{multline}\label{eq.iteration}
    y_0^i+\e\,y_1^i+\e^2y_2^i +\ldots + \e^Ny_N^i =\\
    \f^i(x) +\e\,S^i_1\left(x, \der{g}{y}\left(y_0 +\e\,y_1 +\e^2y_2  +\ldots + \e^{N-1}y_{N-1}\right)\right)  \\
    \shoveright{+\e^2S^i_2\left(x, \der{g}{y}\left(y_0 +\e\,y_1 +\e^2y_2  +\ldots + \e^{N-2}y_{N-2}\right)\right)+ \ldots }  \\
   + \e^N S^i_N\left(x, \der{g}{y}\,(y_0)\right) \quad \mod \e^{N+1}\,.
\end{multline}
Here
\begin{equation*}
    S^i_r(x,q)=(-1)^{\itt}\,\der{S_{r+1}}{q_i}\,(x,q)=\frac{1}{r!}\,\new{S^{ii_1\ldots i_{r}}}(x)q_{i_r}\ldots q_{i_1}\,,
\end{equation*}
where
\begin{equation*}
    S_r(x,q)= \frac{1}{r!}\,\new{S^{i_1\ldots i_{r}}}(x)q_{i_r}\ldots q_{i_1}\,
\end{equation*}
is the $r$th homogeneous term in the expansion~\eqref{eq.sexpanded}. To find $y_N$ from~\eqref{eq.iteration}, we expand the r.h.s. to order $N$ in $\e$ and notice that the terms of order  $\leq N-1$ cancel automatically with the terms at the l.h.s. (since $y_1$, \ldots, $y_{N-1}$ were defined by exactly the same relation at the previous steps). Therefore, it is sufficient to collect the terms of order $N$ at the r.h.s. and upon division by $\e^N$ this gives $y_N$. By the way we see that $y_N$ depends on the derivatives of $g$ to order $\leq N$ (evaluated at $y_0=\f(x)$). We may write symbolically
\begin{equation}\label{eq.phigexpand}
    \f_g= \f + \f_1[g]+ \f_2[g] +\ldots + \f_N[g]+\ldots
\end{equation}
as a functional formal power series in $g$, where each term $\f_N[g]$ is of order $N$ in $g$ and $\f_N^i[g]=y^i_N$   defined by the above procedure. Hence
\begin{equation}\label{eq.fsymb}
    \Phi^*[g] = \f_g^*g+  S\!\left(\!x,\f_g^*\der{g}{y}\right) - \f_g^i\cdot\f_g^*\p_i{g}  \,,
\end{equation}
which justifies the claim about the form of the expansion~\eqref{eq.phiexpanded}.

To summarize, we can now supplement Theorem-Definition~\ref{thm.one} by saying that we should consider the generating function $S(x,q)$ specifying a canonical relation $\F$ as a formal power series in $q_i$; thus the relation  $\F$ is itself formal. The operation $\F^*$ is defined by~\eqref{eq.fsymb} via the above iterative procedure and is therefore a \emph{formal} mapping between the spaces of functions.
Also, the generating function $S(x,q)$ is regarded as a part of   structure, which eliminates questions about a choice of `constants of integration'. This is the first of the clarifications promised in the proof of Theorem-Definition~\ref{thm.one}. Another clarification   concerns the (in)dependence of the procedure defining  $\F^*$ of a choice of coordinates and will follow shortly.

In Example~\ref{ex.linear}, we actually computed the derivative of the mapping $\Phi^*$ at $g\equiv 0$. It is possible to find the derivative of $\Phi^*$ at an arbitrary point.

\begin{theorem} \label{thm.deriv}
The derivative of the formal mapping of functional manifolds
\begin{equation*}
    \Phi^*\co \funn(M_2)\to \funn(M_1)\,
\end{equation*}
at a point $g\in \funn(M_2)$ is given by the formula:
\begin{equation*}
    (T\Phi^*)[g]=\f_g^*\,,
\end{equation*}
where
\begin{equation*}
    \f_g^*\co \fun(M_2)\to \fun(M_1)
\end{equation*}
is the usual pullback with respect to the   map $\f_g\co M_1\to M_2$    defined by $g$.
\end{theorem}
\begin{proof} Consider a variation of a point $g\in \funn(M_2)$, $g_{\e}(y)=g(y)+\e u(y)$. We need to find the corresponding $f_{\e}=\Phi^*[g_{\e}]$. We shall denote by $q_i^{\e}$ and $y^i_{\e}$ the solutions of~\eqref{eq.q} and  ~\eqref{eq.y} for the perturbed function $g_{\e}$, and by symbols without $\e$, the corresponding non-perturbed objects. We may write $y^i_{\e}=y^i+\e y^i_1$ and $q_i^{\e}=q_i+\e q_{1i}$\,. By substituting into~\eqref{eq.f}, we obtain
\begin{multline*}
    f_{\e}(x)=S(x,q+\e q_1) -(y^i+\e y^i_1)(q_i+\e q_{1i})+ g_{\e}(y+\e y_1)=\\
    S(x,q) +  \e\, q_{1i} \der{S}{q_i}\,(x,q)
    -y^i q_i
    -(-1)^{\tilde\e\itt}\e y^iq_{1i}
    -\e \,y^i_1 q_i
    +
    g(y)
    +\e y_1^i \der{g}{y^i}\,(y)
+    \e\,u(y)=\\
S(x,q) +  \e\, q_{1i} (-1)^{\itt}y^i
    -y^i q_i
    -(-1)^{\tilde\e\itt}\e (-1)^{\itt(\itt+\tilde\e)}q_{1i}y^i
    -\e \,y^i_1 q_i
    +
    g(y)
    +\e y_1^i q_i
+    \e\,u(y)=\\
S(x,q)  -y^i q_i +
    g(y) +    \e\,u(y)=
    f(x) +    \e\,u(y)\,.
\end{multline*}
Note that $y=\f_g(x)$\,. Therefore for a perturbation $\e\,u$, $u\in\fun(M_2)$, of   $g\in\funn(M_2)$, the corresponding perturbation of $\Phi^*[g]\in\funn(M_1)$ is $\e\,\f_g^*u$, where $\f_g^*u\in\fun(M_1)$\,.
\end{proof}

Finally, let us turn to the question of a transformation law of the generating function $S(x,q)$ under a change of coordinates. Geometrically, we have a Lagrangian submanifold   $\Phi\subset T^*M_1\times (-T^*M_2)$, which in given coordinates on $M_1$ and $M_2$ is described by~\eqref{eq.relr2} where $S=S(x,q)$ is a function of the variables  $x^a$ and $q_i$ (the coordinates on the base of $T^*M_1$ and the standard fiber of $T^*M_2$, respectively). Suppose we change coordinates:
\begin{equation}\label{eq.coordchange}
    x^a=x^a(x')\,, \ p_a=\der{x^{a'}}{x^a}\,p_{a'}\,, \ y^i=y^i(y')\,, \ q_i=\der{y^{i'}}{y^i}\,q_{i'}\,.
\end{equation}
We need to find a new function $S'=S'(x',q')$ of the   variables $x^{a'}, q_{i'}$ such that in the new coordinates $x^{a'}, p_{a'}, y^{i'}, q_{i'}$ on $T^*M_1\times (-T^*M_2)$ our Lagrangian submanifold   $\Phi$ is specified by the  equations of the same form:
\begin{equation}\label{eq.neweqphi}
    \ p_{a'}=\der{S'}{x^{a'}}\,(x',q')\,, \ y^{i'}=(-1)^{\itt'}\der{S'}{q_{i'}}\,(x',q')\,.
\end{equation}
Since local coordinates on $M_1$ and $M_2$ transform independently, the questions concerning the behavior  of $S$ w.r.t. transformations of $x^a$ and $y^i$ are separate. The behavior w.r.t.  $x^a$ is not problematic: it is easy to see that w.r.t. these variables $S$ can be viewed as representing a genuine function on $M_1$, and so one has to simply perform a substitution in the arguments, $x^i=x^i(x')$. The real problem is with   transformations of coordinates  on $M_2$. The solution is given by the following  statement. (Note that   generating functions are generally defined up to   constants, but we shall give a transformation law for $S$ without such an ambiguity.)

\begin{theorem} \label{thm.translaws}
The `new' generating function $S'(x',q')$ is given the formula
\begin{equation}\label{eq.news}
    S'(x',q')=S(x,q) - y^iq_i +y^{i'}q_{i'}\,,
\end{equation}
where $x^a$,  $q_i$, $y^i$ and $y^{i'}$   are determined from the equations
\begin{equation} \label{eq.newsadd}
   y^{i'}=y^{i'}(y)\,, \quad y^i=(-1)^{\itt}\der{S}{q_i}(x,q)\,, \quad q_i=\der{y^{i'}}{y^i}(y)\,q_{i'}\,, \quad x^a=x^a(x')\,.
\end{equation}
\end{theorem}
\begin{proof} Differentiate both sides of~\eqref{eq.news}\,:
\begin{multline*}
    dS'= dS -dy^iq_i-(-1)^{\itt}y^idq_i +dy^{i'}q_{i'}+(-1)^{\itt'}y^{i'}dq_{i'} \\
    =dS  -(-1)^{\itt}y^idq_i  +(-1)^{\itt'}y^{i'}dq_{i'}=  
    dx^a\,\der{S}{x^a}+dq_i\,\der{S}{q_i}- (-1)^{\itt}dq_iy^i  +(-1)^{\itt'}dq_{i'}y^{i'}\,,
\end{multline*}
and on the submanifold $\Phi$,
\begin{equation*}
    dS'=
    dx^a\,p_a+(-1)^{\itt}dq_i\,y^i- (-1)^{\itt}dq_iy^i  +(-1)^{\itt'}dq_{i'}y^{i'}=
    dx^{a'}\,p_{a'}+(-1)^{\itt'}dq_{i'}y^{i'}\,,
\end{equation*}
which gives~\eqref{eq.neweqphi} as desired. To properly make use of formula~\eqref{eq.news}, one has to express all the variables at the r.h.s. of it, i.e., $x^a$, $q_i$, $y^i$ and $y^{i'}$ in terms of the variables at the l.h.s., i.e., $x^{a'}$ and $q_{i'}$\,. For $x^a$, we simply substitute $x^a=x^a(x')$. We also substitute $y^{i'}=y^{i'}(y)$ and use the standard transformation law for the momentum variables $q_i$, expressing them via $q_{i'}$ and $y^i$.  The rest is subtler: for determining $y^i$  we have a system of coupled equations
\begin{equation*}
    y^i=(-1)^{\itt}\der{S}{q_i}(x,q)\,, \quad q_i=\der{y^{i'}}{y^i}(y)\,q_{i'}\,,
\end{equation*}
which gives 
\begin{equation*}
    y^i=(-1)^{\itt}\der{S}{q_i}\bigl(x,\der{y'}{y}(y)q'\bigr)\,,  
\end{equation*}
from where $y$ is expressed as a function of $x$ and $q'$ by an iterative procedure similar to that defining the map $\f_g$ above. The result is a formal power expansion in $q'$.
\end{proof}

Formula~\eqref{eq.news} can be read as the composition of   three transformations: the `direct' Legendre transform from $q_i$ to $y^i$, the substitution $y^i=y^i(y')$, and the `inverse' Legendre transform from $y^{i'}$ to $q_{i'}$.\footnote{Note an analogy with pseudodifferential operators: the direct Fourier transform, then a multiplication operator, and then the inverse Fourier transform. It is not a random analogy because  the Legendre transform can be seen as the `classical limit' of the Fourier transform. We can treat formula~\eqref{eq.f} in a similar way. Compare with~\cite{tv:oscil, tv:qumicro}.} Namely, we pass from $S(q)$ (here the dependence on $x$ is suppressed) to $S^*(y)$,
\begin{equation*}
    S^*(y)= y^iq_i- S(q)\,,
\end{equation*}
where $q_i$ is expressed from
\begin{equation*}
    y^i=(-1)^{\itt}\der{S}{q_i}\,(q)\,.
\end{equation*}
Then we substitute to obtain ${S^*}'(y'):=S^*\bigl(y(y')\bigr)$. Finally, we pass from ${S^*}'$ to $S'= {({S^*}')}^*$,
\begin{equation*}
    S'(q')= y^{i'}q_{i'}- S^*\bigl(y(y')\bigr)\,,
\end{equation*}
where $y^{i'}$ is expressed from
\begin{equation*}
    q_{i'}= \der{}{y^{i'}}\,S^*\bigl(y(y')\bigr)\,.
\end{equation*}
Assembled together, these steps give  equation~\eqref{eq.news}.
The possibility to make the Legendre transform from $q$ to $y$   puts a restriction on the generating function $S$ (non-degeneracy in $q$). Such a restriction is not satisfied, for example, by $S$ corresponding to a smooth map $M_1\to M_2$.  However, the restriction disappears for the  composite transformation $S(q)\mapsto S'(q')$, because the two Legendre transforms compensate each other in a way. (It is illuminating to see how the inverse matrix $S_{ij}$  eventually disappears from the final answer after one initially  assumes the non-degeneracy of the quadratic form $\new{S^{ij}}q_jq_i$ in the expansion of $S$ so to be able to apply the   Legendre transform.)  Theorem~\ref{thm.translaws} does not require any  non-degeneracy from $S$.

\begin{remark} The transformation law for $S$ given by~\eqref{eq.news} and~\eqref{eq.newsadd} satisfies the \emph{cocycle condition}, as one can immediately see: if $S'(x',q')$ is expressed from $S(x,q)$ by~\eqref{eq.news},\eqref{eq.newsadd}, and $S''(x'',q'')$ is expressed from $S'(x',q')$ by the same  formulas  (with the necessary replacements), then the composite expression of $S''(x'',q'')$ via $S(x,q)$ coincides with the direct expression given by these formulas. This makes it possible to consider generating functions $S(x,q)$ (defined as power series in $q$) as geometric objects on $M_1\times M_2$.
\end{remark}

\begin{example} Suppose the `old' generating function $S(x,q)$ is given by the expansion~\eqref{eq.sexpanded}. Under a change of coordinates~\eqref{eq.coordchange}, the `new' generating function $S'(x',q')$ has the expansion
\begin{equation}\label{eq.expannews}
    S'(x',q')=S_0\bigl(x(x')\bigr) + \f^{i'}(x')\,q_{i'} +   \frac{1}{2}\,\new{S^{i'\!j'}}(x')\,q_{j'}q_{i'}+O(|q'|^3) ,
\end{equation}
where
\begin{equation}\label{eq.newphii}
    \f^{i'}(x')=y^{i'}\bigl(\f(x(x'))\bigr)\,,
\end{equation}
and
\begin{equation}\label{eq.newphiij}
    \new{S^{i'\!j'}}(x')= (-1)^{\itt(\itt'+1)}\der{y^{i'}}{y^i}\,\bigl(\f'(x')\bigr)\,\new{S^{ij}}\bigl(x(x')\bigr)\,  \der{y^{j'}}{y^j}\,\bigl(\f'(x')\bigr)\,.
\end{equation}
This can be obtained by a patient calculation along the lines above, which we leave to the pleasure of the reader. 
Note that~\eqref{eq.newphii} is just the expression of the map $\f\co M_1\to M_2$ in  new coordinates on $M_1$ and $M_2$, and in equation~\eqref{eq.newphiij} one recognizes the tensor law on $M_2$ at a point $\f(x)$. Non-tensor transformations  depending,  in particular,  on higher derivatives of a coordinate transformation on $M_2$  appear in the higher order terms of $S$.
\end{example}

The following statement is a direct consequence of the definition of $\F^*$ and the transformation law given by Theorem~\ref{thm.translaws}. It deserves the name of a theorem because of its importance.

\begin{theorem} 
Suppose $g'=g'(y')$ is the expression of an even function $g=g(y)$ in new coordinates on $M_2$, i.e., $g'(y')=g(y(y'))$, and  $S'(x',q')$ is the expression of a generating function $S(x,q)$ in new coordinates on $M_1$ and  $M_2$ according to the transformation law~\eqref{eq.news},\eqref{eq.newsadd}.  Let the function  $f'=f'(x')$ be obtained from $g'$, $S'$  and the function $f=f(x)$ be obtained from $g$, $S$,  as  the generalized pullbacks  (in coordinates $x'$, $y'$ and $x$, $y$, respectively). Then  the function $f'=f'(x')$ is the expression of the function $f=f(x)$ in the new coordinates on $M_1$, i.e., $f'(x')=f(x(x'))$.
\end{theorem}
\begin{proof} We are given that
\begin{equation*}
    f'(x')= g'(y') + S'(x',q')-y'q'\,,
\end{equation*}
where
\begin{equation*}
    y^{i'}=(-1)^{i'}\der{S'}{q_{i'}}(x',q')\,, \quad q_{i'}=\der{g'}{y'}(y')\,.
\end{equation*}
Also $g'(y')=g(y(y'))$, for an invertible change of variables $y'=y'(y)$, and
\begin{equation*}
    S'(x',q')=S(x,q)-yq+y'q'\,,
\end{equation*}
where
\begin{equation*}  
   y^{i'}=y^{i'}(y)\,, \quad y^i=(-1)^{\itt}\der{S}{q_i}(x,q)\,, \quad q_i=\der{y^{i'}}{y^i}(y)\,q_{i'}\,, \quad x^a=x^a(x')\,.
\end{equation*}
Note that this transformation law for $S$ implies $y^{i'}=(-1)^{i'}\lder{S'}{q_{i'}}(x',q')$ if $y^{i}=(-1)^{i}\lder{S}{q_{i}}(x,q)$. Hence we can `compose' the formulas for $f'$ and $S'$ to obtain
\begin{equation*}
   f'(x')=  g(y(y'))+S(x,q)-yq+y'q'-y'q'=g(y)+S(x,q)-yq\,,
\end{equation*}
where at the r.h.s.
\begin{equation*}
    y^{i}=(-1)^{i}\der{S}{q_{i}}(x,q)\,, \quad q_{i}=\der{g}{y}(y)\,, 
\end{equation*}
and $x=x(x')$. This is exactly the equality   $f'(x')=f(x(x'))$, as claimed.
\end{proof}

To summarize, we may say that a generalized pullback  $\F^*$, a formal mapping of function spaces  defined  initially in local coordinates, \emph{is independent of a choice of coordinates}. This finishes with all questions of substantiation.

We leave out    discussion of compositions of formal canonical relations $\F$ given by formal generating functions $S(x,q)$. One should expect, in view of the analysis performed above, that they form what can be regarded as a  \emph{formal category}\footnote{The underlying category, for which this formal category is a formal neighborhood, being the semi-direct product of the usual category of smooth supermanifolds and their smooth maps with algebras of smooth functions.} and the usual formula
 \begin{equation} \label{eq.composit}
    (\Phi_1\circ \Phi_2)^*=\Phi_2^*\circ \Phi_1^*
 \end{equation}
holds, so  our analog of pullbacks gives a nonlinear representation of (the dual of) this formal category. (Formula~\eqref{eq.composit} should basically follow from the associativity of composition of relations.) These questions are considered fully in our forthcoming work~\cite{tv:microformal}.

\begin{remark} We have worked so far with formal objects (power series). To extend consideration to non-formal objects  may be possible but may require more work. Our central equation~\eqref{eq.y}   defines a map $\f_g\co M_1\to M_2$ associated with a canonical relation $\Phi$ and a function $g$. We showed above how to solve it by iterations so to obtain a power series solution. At the same time, one can imagine that a Banach contraction mapping argument can be used for obtaining a non-formal solution of~\eqref{eq.y} in a neighborhood of the zero section.  Since such a neighborhood is   unspecified, a neat formulation would be to replace it by a \emph{germ}. So the options are to work on a formal level with power series (infinite jets) or   with germs.  Considering germs of symplectic manifolds at Lagrangian submanifolds is Weinstein's idea dating back to~\cite{weinstein:sympl71}. If we follow this direction, our work will immediately meet the
recent body of works on ``symplectic microgeometry''  such as~\cite{weinstein:symplectic-geometries}, \cite{cattaneo-dherin-weinstein:one, cattaneo-dherin-weinstein:two, cattaneo-dherin-weinstein:three}. 
\end{remark}

\section{Hamilton--Jacobi vector fields}

Consider a Hamiltonian function $H\in\fun(T^*M)$, which can be even or odd. We write $H=H(x,p)$, as usual. To such a function we assign a   vector field $\X_H$ on the infinite-dimensional manifold  $\funn(M)$, as follows: for each $f\in \funn(M)$, the variation of $f$ is given by
\begin{equation}\label{eq.xh}
    f\mapsto f_{\e}=f  +\e \X_H[f]\,, \ \text{where} \  f_{\e}(x)=f(x)+ \e H\Bigl(x,\der{f}{x}\,(x)\Bigr)\,.
\end{equation}
Here $\e^2=0$ and $\tilde\e=\Ht$. The parity of the vector field $\X_H$ is the same as the parity of $H$.
In standard terminology used in field theory or integrable systems, the vector field $\X_H$ is a `first-order local vector field' on the space of functions. It can be written in terms of variational derivatives as
\begin{equation}\label{eq.xh1}
    \X_H = \new{(-1)^{\Ht m}\int\limits_{M^{n|m}}}\!\!  Dx\; H\Bigl(x,\der{f}{x}\,(x)\Bigr)\var{}{f(x)}\,.
\end{equation}
\new{(The sign   is required for   linearity.)}

The differential equation defining the flow of the vector field $\X_H$ on the manifold $\funn(M)$ is   a  Hamilton--Jacobi equation.\footnote{From arbitrary first order local vector fields, the vector fields $\X_H$ are distinguished by the  dependence only on the values of the  derivative but not   the function itself. This is precisely what distinguishes the Hamilton--Jacobi equations among arbitrary first order partial differential equations.} It takes the familiar form
\begin{equation}\label{eq.hamjacev}
    \der{f}{t}= H\Bigl(x,\der{f}{x}\Bigr)
\end{equation}
when $H$ is even. Here   the time variable $t$ in~\eqref{eq.hamjacev} is also even. (In~\eqref{eq.hamjacev}, a function $f$ depends on $t$ in addition to   $x$, so to give a curve in  $\funn(M)$.) For an odd $H$, the corresponding  Hamilton--Jacobi equation takes the form
\begin{equation}\label{eq.hamjacod}
    Df\equiv\left(\der{}{\tau}+\tau \der{}{t}\right)f= H\Bigl(x,\der{f}{x}\Bigr)\,,
\end{equation}
with  two time variables, even $t$ and odd $\tau$. (The operator $D$ at the l.h.s. of~\eqref{eq.hamjacod} squares to   $\lder{}{t}$.) 

\begin{theorem} \label{prop.commutofxh}
\label{thm.commut}
For arbitrary Hamiltonians $H$ and $F$,
\begin{equation}\label{eq.commutofxh}
    \left[\X_H,\X_F\right]=-\X_{\left(H,F\right)}\,,
\end{equation}
where the bracket at the l.h.s. is the commutator of vector fields on the infinite-dimensional manifold $\funn(M)$ and the bracket at the r.h.s. is the canonical Poisson bracket on $T^*M$.
\end{theorem}
(The minus sign in~\eqref{eq.commutofxh} is of course completely inessential and depends   on  conventions.)
\begin{proof} Direct calculation, but   still worth giving here.  For calculating the commutator, we start from a point $f_0\in \funn(M)$ and apply to it successively infinitesimal shifts along the vector fields $\X_H$ and $\X_F$. First we arrive at $f_1$, where
\begin{equation*}
    f_1(x)=f_0(x)+\e H\Bigl(x,\der{f_0}{x}\Bigr)\,
\end{equation*}
and $\e^2=0$. Then we arrive at $f_2$, where
\begin{multline*}
    f_2(x)=f_1(x) +\h F\Bigl(x,\der{f_1}{x}\Bigr)=\\
    f_0(x)+\e H\Bigl(x,\der{f_0}{x}\Bigr) + \h F\left(x,\der{f_0}{x} +\der{}{x}\e H\Bigl(x,\der{f_0}{x}\Bigr)\right)=\\
    f_0(x)+\e H\Bigl(x,\der{f_0}{x}\Bigr) + \h F\Bigl(x,\der{f_0}{x}\Bigr)  +\h\der{}{x^a}\e H\Bigl(x,\der{f_0}{x}\Bigr)\cdot \der{F}{p_a}\Bigl(x,\der{f_0}{x}\Bigr)\,
\end{multline*}
and $\h^2=0$.
Next we arrive at $f_3$, where
\begin{multline*}
    f_3(x)=f_2(x) -\e  H\Bigl(x,\der{f_2}{x}\Bigr)=\\
    f_0(x)+\e H\Bigl(x,\der{f_0}{x}\Bigr) + \h F\Bigl(x,\der{f_0}{x}\Bigr)  +\h\der{}{x^a}\e H\Bigl(x,\der{f_0}{x}\Bigr)\cdot \der{F}{p_a}\Bigl(x,\der{f_0}{x}\Bigr)
    -\\
    \e  H\Bigl(x,\der{f_0}{x}  + \der{}{x}\h F\Bigl(x,\der{f_0}{x}\Bigr)\Bigr)=\\
    f_0(x)+\e H\Bigl(x,\der{f_0}{x}\Bigr) + \h F\Bigl(x,\der{f_0}{x}\Bigr)  +\h\der{}{x^a}\e H\Bigl(x,\der{f_0}{x}\Bigr)\cdot \der{F}{p_a}\Bigl(x,\der{f_0}{x}\Bigr)
    -\\
    \e  H\Bigl(x,\der{f_0}{x}\Bigr) -\e \der{}{x^a}\h F\Bigl(x,\der{f_0}{x}\Bigr)\cdot \der{H}{p_a}\Bigl(x,\der{f_0}{x}\Bigr)=\\
    f_0(x) + \h F\Bigl(x,\der{f_0}{x}\Bigr)  + 
    \h\der{}{x^a}\e H\Bigl(x,\der{f_0}{x}\Bigr)\cdot \der{F}{p_a}\Bigl(x,\der{f_0}{x}\Bigr)
    - \\
     \e \der{}{x^a}\h F\Bigl(x,\der{f_0}{x}\Bigr)\cdot \der{H}{p_a}\Bigl(x,\der{f_0}{x}\Bigr)\,.
\end{multline*}
Finally we arrive at $f_4$, where
\begin{multline*}
    f_4(x)=f_3(x) -\h  F\Bigl(x,\der{f_3}{x}\Bigr)= 
    f_0(x)   +\h\der{}{x^a}\e H\Bigl(x,\der{f_0}{x}\Bigr)\cdot \der{F}{p_a}\Bigl(x,\der{f_0}{x}\Bigr)
    -\\
    \e \der{}{x^a}\h F\Bigl(x,\der{f_0}{x}\Bigr)\cdot \der{H}{p_a}\Bigl(x,\der{f_0}{x}\Bigr)= 
     f_0(x)   + 
     \h\e\left((-1)^{\at\Ht}\der{}{x^a} H\Bigl(x,\der{f_0}{x}\Bigr)\cdot \right.\\
    \left. \der{F}{p_a}\Bigl(x,\der{f_0}{x}\Bigr) - (-1)^{\Ft(\at+\Ht)}  \der{}{x^a} F\Bigl(x,\der{f_0}{x}\Bigr)\cdot \der{H}{p_a}\Bigl(x,\der{f_0}{x}\Bigr)\right)\,.
\end{multline*}
By examining the differential expression in the big bracket, we observe that the terms of the first order in $f_0$ assemble to
\begin{multline*}
    (-1)^{\at\Ht}\der{H}{x^a}\, \der{F}{p_a}
    -(-1)^{\Ft(\at+\Ht)}  \der{F}{x^a}\, \der{H}{p_a}= \\
    -(-1)^{\at\Ht}\!\left(\!(-1)^{\at}\,\der{H}{p_a}\, \der{F}{x^a}- \der{H}{x^a}\, \der{F}{p_a}\right)=  
    -\left(H,F\right)\,,
\end{multline*}
the Poisson bracket of $H$ and $F$, evaluated at $\bigl(x,\der{f_0}{x}\bigr)$\,. At the same time, the terms of the second order in $f_0$ are
\begin{equation*}
    (-1)^{\at\Ht}\dder{f_0}{x^a}{x^b}\, \der{H}{p_b}\,\der{F}{p_a}
    -(-1)^{\Ft(\at+\Ht)}  \dder{f_0}{x^a}{x^b}\, \der{F}{p_b}\,\der{H}{p_a}
\end{equation*}
and we can observe that they cancel by the symmetry of second partial derivatives. Hence
\begin{equation*}
    f_4(x)= f_0-\h\e\,(H,F)\Bigl(x,\der{f_0}{x}\Bigr)\,,
\end{equation*}
as claimed.
\end{proof}

\begin{coro} Let $Q$ be an odd Hamiltonian and $(Q,Q)=-2H$, so $H$ is even. Then the solution of the Hamilton--Jacobi equation for $Q$,
\begin{equation*}
    Df= Q\Bigl(x,\der{f}{x}\Bigr)\,,
\end{equation*}
is given by $f(t,\tau)=f_0(t)+\tau f_1(t)$, where $f_0$ is the solution of the usual Hamilton--Jacobi equation for $H$,
\begin{equation*}
    \der{f_0}{t}= H\Bigl(x,\der{f_0}{x}\Bigr)\,,
\end{equation*}
and $f_1 =Q\Bigl(x,\der{f_0}{x}\Bigr)$\,. In particular, if $(Q,Q)=0$, then the Hamilton--Jacobi equation for $Q$ reduces to
\begin{equation*}
    \der{f}{\tau}= Q\Bigl(x,\der{f}{x}\Bigr)\,,
\end{equation*}
and its solution is just an `odd shift'\,:  $f= f_0+ \tau Q\Bigl(x,\der{f_0}{x}\Bigr)$\,.
\end{coro}

We shall refer to the vector fields on the infinite-dimensional manifold $\funn(M)$ of the form $\X_H$ as to the \emph{Hamilton--Jacobi vector fields}.

Consider an arbitrary relation $R\subset T^*M_1\times T^*M_2$. We say that Hamiltonians $H_1\in\fun(T^*M_1)$ and $H_2\in\fun(T^*M_2)$ are \emph{$R$-related} if $p_1^*H_1=p_2^*H_2$, where $p_i$, $i=1,2$, are the restrictions of the canonical projections on $T^*M_i$. This terminology extends the classical notion of $\f$-related vector fields as shown by the following example.

\begin{example} Suppose $R=R_{\f}$ corresponds to a smooth map $\f\co M_1\to M_2$ as in Example~\ref{ex.rphi}. Then the condition that $H_1=H_1(x,p)$ and $H_2=H_2(y,q)$ are $R$-related amounts to
\begin{equation*}
    H_1\left(\!x,\der{\f}{x}(x)\,q\right)=H_2\bigl(\f(x),q\bigr)\,.
\end{equation*}
In particular, if $H_1(x,p)=X^a(x)p_a$ and $H_2(y,q)=Y^i(y)q_i$ correspond to vector fields $X\in\Vect(M_1)$ and $Y\in\Vect(M_2)$, we recognize the familiar condition
\begin{equation*}
    X^a(x)\,\der{\f^i}{x^a}  = Y^i\bigl(\f(x)\bigr)\,,
\end{equation*}
i.e.,
that the vector fields $X$ and $Y$ are $\f$-related.
\end{example}

Suppose  there is a canonical relation $\Phi\subset T^*M_1\times (-T^*M_2)$ of the form~\eqref{eq.relr2}. Consider the pullback $\Phi^*\co \funn(M_2)\to \funn(M_1)$\,. 

\begin{theorem} \label{thm.main}
If  Hamiltonians $H_1\in\fun(T^*M_1)$ and $H_2\in\fun(T^*M_2)$  are $\Phi$-related, then the Hamilton--Jacobi vector fields $\X_{H_2}\in \Vect(\funn(M_2))$ and $\X_{H_1}\in \Vect(\funn(M_1))$   are $\Phi^*$-related.
\end{theorem}
\begin{proof} The condition that two vector fields are related by a smooth map means that the map intertwines the corresponding infinitesimal shifts. We shall check that for the vector fields $\X_{H_2}$ and $\X_{H_1}$. Note that the condition that $H_1$ and $H_2$ are $\Phi$-related reads:
\begin{equation}\label{eq.hamrel}
    H_1(x,p)=H_2(y,q)\quad
\text{for}\quad
   p_a=\der{S}{x^a}\,(x,q) \ \text{and}\  y^i=(-1)^{\itt}\,\der{S}{q_i}\,(x,q)\,.
\end{equation}
Take an arbitrary $g\in \funn(M_2)$ and apply to it the infinitesimal shift along $\X_{H_2}$. We obtain
\begin{equation*}
    g_{\e}(y)=g(y)+\e\,H_2\Bigl(y,\der{g}{y}\Bigr)\,.
\end{equation*}
Apply to the result the map $\Phi^*$. %
By Theorem~\ref{thm.deriv},
\begin{equation*}
    \Phi^*[g_{\e}]  =\Phi^*[g]+\e\,\f_g^*\left(H_2\Bigl(y,\der{g}{y}\Bigr)\right)\,.
\end{equation*}
Recall that $\f_g^*$ simply means that $y^i$ should be found from the equation
\begin{equation}\label{eq.yy}
    y^{i}=(-1)^{\itt}\,\der{S}{q_i}\Bigl(x,\der{g}{y}\Bigr)\,.
\end{equation}
In the opposite direction, apply first $\Phi^*$ to $g$ to obtain $\Phi^*[g]$ and then apply to it the infinitesimal shift along $\X_{H_1}$. We arrive at
\begin{equation*}
    \Phi^*[g] +\e H_1\Bigl(x,\der{\Phi^*[g]}{x}\Bigr)\,.
\end{equation*}
Denote $\Phi^*[g]=:f$. To calculate the derivative in the argument,   write
\begin{equation*}
    f(x)=S(x,q)-y^iq_i+g(y)\,,
\end{equation*}
where
\begin{equation*}
    q_i=\der{g}{y}\quad  \text{and}\ \ y^{i}=(-1)^{\itt}\,\der{S}{q_i}\Bigl(x,\der{g}{y}\Bigr)\,,
\end{equation*}
so
\begin{multline*}
    \der{f}{x^a}=\der{S}{x^a}\Bigl(x,\der{g}{y}\Bigr)+
    \der{q_i}{x^a}\,\der{S}{q_i}\Bigl(x,\der{g}{y}\Bigr)-\der{y^i}{x^a}\,q_i-
    (-1)^{\itt\at}y^i\,\der{q_i}{x^a}+\der{y^i}{x^a}\der{g}{y^i}=\\
    \der{S}{x^a}\Bigl(x,\der{g}{y}\Bigr)+(-1)^{\itt}\der{q_i}{x^a}\,y^i-\der{y^i}{x^a}\,q_i-
    (-1)^{\itt\at+\itt(\at+\itt)}\,\der{q_i}{x^a}\,y^i+\der{y^i}{x^a}\,q_i=
    \der{S}{x^a}\Bigl(x,\der{g}{y}\Bigr)\,.
\end{multline*}
Therefore, to prove our statement, we need to compare the infinitesimal increments of $f(x)$, $f=\Phi^*[g]$,  given  in one case by
\begin{equation*}
    H_2\Bigl(y,\der{g}{y}\Bigr)
\end{equation*}
(after dropping $\e$) and in the other case by
\begin{equation*}
    H_1\Bigl(x,\der{S}{x^a}\Bigl(x,\der{g}{y}\Bigr)\Bigr)\,,
\end{equation*}
where in both cases $y^i$ is obtained from~\eqref{eq.yy}\,.
We see that the equality in question
\begin{equation*}
    H_1\Bigl(x,\der{S}{x^a}\Bigl(x,\der{g}{y}\Bigr)\Bigr)=H_2\Bigl(y,\der{g}{y}\Bigr)
\end{equation*}
follows from~\eqref{eq.hamrel}, which is valid for all $q$, in particular $q=\lder{g}{y}$\,.
\end{proof}

This theorem may be seen as the main statement of our paper.

\section{Application to homotopy algebras and algebroids}

Let us recall   some information concerning $L_{\infty}$-algebras. We shall use the higher derived bracket construction~\cite{tv:higherder}. A vector space $L$ together with an infinite sequence of odd symmetric multilinear operations (`brackets')
\begin{equation*}
    \underbrace{L\times \ldots \times L}_{r\, \text{times}} \to L\,,
\end{equation*}
where $r=0,1,2,3,\ldots\ $, is called an \emph{$L_{\infty}$-algebra} if the brackets satisfy the sequence of `higher  Jacobi identities'
\begin{equation}\label{eq.jac}
    J_n(v_1,\ldots,v_n)=0\,,
\end{equation}
for all $n=0,1,2,3,\ldots\ $,\footnote{More precisely,  this is an $L_{\infty}$-algebra ``in the symmetric version''. In the original terminology~\cite{lada:stasheff}, an $L_{\infty}$-algebra or `strongly homotopy Lie' algebra has antisymmetric brackets of alternating parities, namely, brackets with an even number of arguments being even and with an odd number of arguments, odd. These two notions transform to each other by the parity reversion of the underlying space, see~\cite{tv:higherder}.} where $J_n(v_1,\ldots,v_n)$ denotes the $n$th  \emph{Jacobiator}  of the brackets   defined by
\begin{equation}\label{eq.jacobiator}
    J_n(v_1,\ldots,v_n):= \sum_{\parbox{1.2cm}{\small $\scriptstyle \;k,\ell\,\geq\, 0\vspace{-4pt}\\k+\ell\,=n$}}\! \sum_{\text{$(k,\ell)$-shuffles}}\!(-1)^{\a}\{\{v_{\s(1)},\ldots,v_{\s(k)}\},v_{\s(k+1)},\ldots,v_{\s(k+\ell)}\}\,.
\end{equation}
(Here  the sign $(-1)^{\a}$ is the usual Koszul sign depending on the parities of permuted  arguments, e.g., $(-1)^0=+1$,  if all $v_i$ are even.) 

A sequence of symmetric multilinear operations of a given parity on a vector space $L$  can be assembled into a  formal vector field on the corresponding `vector supermanifold' $\bL$, where we use boldface for distinction. Conversely, given a vector field $X\in\Vect(\bL)$, a sequence of brackets on $L$ is obtained  as follows~\cite{tv:higherder}:
\begin{equation}\label{eq.highder}
    i_{\{v_1,\ldots,v_r\}} =\left[\ldots\left[\left[X,i_{v_1}\right],i_{v_2}\right],\ldots,i_{v_2}\right](0)\,,
\end{equation}
(evaluation at the origin), where $i_v$ is the constant vector field corresponding to a vector $v\in L$. Suppose a sequence of odd brackets on $L$ corresponds  to an odd vector field $Q\in \Vect (\bL)$. Then the sequence of their  Jacobiators $J_n$ 
corresponds to the even vector field $Q^2=\frac{1}{2}\,[Q,Q]$. (See~\cite{tv:higherder} for a more general statement.) Therefore there is a one-to-one correspondence between $L_{\infty}$-algebra structures on a vector space $L$ and formal homological vector fields on $\bL$. It is known that the language of homological vector fields is the most efficient way of working with $L_{\infty}$-algebras (see, e.g.~\cite{kontsevich:quant, kontsevich:quantlmp}). In particular, an \emph{$L_{\infty}$-morphism} from an $L_{\infty}$-algebra $L_1$ to an $L_{\infty}$-algebra $L_2$ (in the above description) can be defined as a formal supermanifold map $\f\co \bL_1\to \bL_2$ (in general, nonlinear) such that the corresponding homological vector fields $Q_i\in\Vect(\bL_i)$ are $\f$-related.

We shall apply these general notions to the setup where brackets are introduced on the space of smooth functions on some (super)manifold.

A  Hamiltonian $H\in \fun(T^*M)$ defines a sequence of symmetric brackets on the vector space $\fun(M)$  by the higher derived bracket construction~\cite{tv:higherder}\,:
\begin{equation}\label{eq.brackets}
    \{f_1,\ldots,f_r\}_H:=\left(\ldots\left(\left(H,f_1\right),f_2\right),\ldots,f_r\right)_{|M}\,.
\end{equation}
The parity of these brackets is the same as the parity of $H$.  All brackets~\eqref{eq.brackets} are multiderivations w.r.t. the associative multiplications of functions.
If we expand $H$ as
\begin{equation}\label{eq.ham}
    H(x,p)=H_0(x)+H^a(x)p_a +\frac{1}{2}\,H^{ab}(x)p_bp_a + \frac{1}{3!}\,H^{abc}(x)p_cp_bp_a+\ldots \,,
\end{equation}
with symmetric coefficients $H^{a_1\ldots a_r}$, then
\begin{equation}\label{eq.brackets1}
    \{f_1,\ldots,f_r\}_H= \pm H^{a_1\ldots a_r}(x)\,\p_{a_r}f\ldots \p_{a_1}f\,.
\end{equation}
We shall refer to the Hamiltonian generating a given sequence of brackets as to the   \emph{master Hamiltonian}.
It is natural to ask what is the corresponding vector field on the infinite-dimensional supermanifold $\funn(M)$. The answer is given by the following statement.
\begin{theorem} \label{thm.derhamjac}
The higher derived brackets~\eqref{eq.brackets} generated by   $H\in \fun(T^*M)$ assemble to the Hamilton--Jacobi vector field $\X_H\in\Vect(\funn(M))$\,,
\begin{equation}\label{eq.xh2}
    \X_H = \new{(-1)^{\Ht m}\int\limits_{M^{n|m}}}\!\!  Dx\; H\Bigl(x,\der{f}{x}\,(x)\Bigr)\var{}{f(x)}\,.
\end{equation}
\end{theorem}
\begin{proof} Directly. One needs to apply~\eqref{eq.brackets} to $f_1=\ldots=f_r=f$, for some   even function $f\in \funn(M)$.
\end{proof}

When the master Hamiltonian $H$ is odd, the derived  brackets~\eqref{eq.brackets} are also odd and it is legitimate to ask   whether the Jacobi identities~\eqref{eq.jac} hold for them. As follows from a general theorem~\cite{tv:higherder}, if an odd Hamiltonian $H$   obeys the classical \emph{master equation}
\begin{equation}\label{eq.master}
    (H,H)=0\,,
\end{equation}
then all the Jacobi identities are satisfied for its derived brackets, so the space $\fun(M)$   with these brackets is an $L_{\infty}$-algebra. Considered also with the ordinary multiplication of functions, it is a  {homotopy Schouten algebra}. (By definition, a \emph{homotopy Schouten algebra} or \emph{$S_{\infty}$-algebra} is a commutative associative algebra endowed with an infinite sequence of odd symmetric brackets that satisfy the higher Jacobi identities and also the Leibniz identity in each argument~\cite{tv:higherder}.) A supermanifold $M$ whose algebra of functions is endowed with odd brackets making it a homotopy Schouten algebra will be called a \emph{homotopy Schouten manifold} or an \emph{$S_{\infty}$-manifold}.

\begin{remark} By Theorem~\ref{prop.commutofxh}, for an  odd Hamiltonian $H$ we have
\begin{equation*}
    [\X_H,\X_H]=-\X_{(H,H)}\,.
\end{equation*}
So if $H$ satisfies $(H,H)=0$, then $\X_H^2=0$. This gives a direct proof that such an $H$ generates an $L_{\infty}$-algebra.
\end{remark}

Consider homotopy Schouten  manifolds $M_1$ and $M_2$. Let $H_1$ and $H_2$ be the respective master Hamiltonians. Let $\Phi\subset T^*M_1\times (-T^*M_2)$ be a canonical relation of the form~\eqref{eq.relr2}.
\begin{coro}[From Theorems~\ref{thm.main} and~\ref{thm.derhamjac}] If   master Hamiltonians $H_1$ and $H_2$ are $\Phi$-related, then the formal mapping
of function  supermanifolds
\begin{equation*}
    \Phi^*\co \funn(M_2)\to \funn(M_1)
\end{equation*}
\emph{(in general, nonlinear)} is an $L_{\infty}$-morphism of the corresponding $L_{\infty}$-algebras.
\end{coro}

\begin{example} A very special case is that of  $\Phi$ corresponding to an ordinary map $\f\co M_1\to M_2$, so that the master Hamiltonians are $\f$-related. Then $\f^*\co \fun(M_2)\to \fun(M_1)$ is the usual pullback, hence linear. It gives a  strict  morphism of   $L_{\infty}$-algebras, for which all brackets are preserved separately.
\end{example}

This can also be applied to Lie algebroid theory,  as follows. 

For Lie bialgebroids, a {Lie bialgebroid morphism} $E_1\to E_2$ is defined as a morphism of Lie algebroids $\f\co E_1\to E_2$ such that it is also a Poisson map for the Lie--Poisson brackets induced by the Lie algebroid structures on the dual bundles~\cite{mackenzie:book2005}. The latter condition is also equivalent to $\f$ being a Poisson map for the  Lie--Schouten brackets induced on $\Pi E_1$ and $\Pi E_2$. The conditions of a Lie algebroid morphism and a Poisson map  naturally combine together into one condition that the odd Hamiltonians defining the $QS$-structures (see~\cite{tv:graded}) on $\Pi E_1$ and $\Pi E_2$ are $\f$-related. This is   equivalent to $\f^*\co \fun(\Pi E_2)\to \fun(\Pi E_1)$ being a morphism of differential Schouten algebras. The question arises, what should stand for all that in the homotopy case.

A structure of an $L_{\infty}$-bialgebroid is defined on a vector bundle $E$ by an odd master Hamiltonian $H$ satisfying the master equation $(H,H)=0$. It particular it makes  
the algebra of functions $\fun(\Pi E)$ is a homotopy Schouten algebra. How one should define morphisms of $L_{\infty}$-bialgebroid? We should be looking for constructions leading to  $L_{\infty}$-morphisms of the  algebras of functions. Ordinary morphisms of vector bundles can only lead to  strict morphisms. This is clearly not sufficient. The correct notion should use   nonlinear pullbacks (as can be showed).

\begin{example} \label{ex.highkoszul}
Let a supermanifold $M$ have a homotopy Poisson structure (see, e.g.,~\cite{tv:higherpoisson} and in the  Appendix). (The difference with a homotopy Schouten structure is that the brackets are antisymmetric and have alternating parities, so that the binary bracket is even.) In~\cite{tv:higherpoisson} we showed that it induces the structure of an $L_{\infty}$-algebroid on the cotangent bundle $T^*M$. (This is the analog of the Lie algebroid structure on $T^*M$ for an  ordinary Poisson manifold.) The corresponding sequence of odd brackets on functions on $\Pi TM$ are called the \emph{higher Koszul brackets}. Recall that  functions on $\Pi TM$ are (pseudo)differential  forms on $M$. In the classical situation, there is only the binary  Koszul bracket on forms induced by an ordinary Poisson structure and the pullback w.r.t. the Poisson anchor maps it to the canonical Schouten bracket of multivector fields. In~\cite{tv:higherpoisson}, we posed the problem of extending this picture to the homotopy Poisson case, i.e., to find an $L_{\infty}$-morphism between the higher Koszul brackets and the canonical Schouten bracket. The solution is given by
a certain nonlinear pullback $\Phi^*\co \funn(\Pi TM)\to \funn(\Pi T^*M)$ (see~\cite{tv:microformal}). This question was the departure point of the present work.
\end{example}

Example~\ref{ex.highkoszul} has an abstract form, which is   an $L_{\infty}$ version of  `triangular Lie bialgebroids'  of Mackenzie--Xu~\cite{mackenzie:bialg} and in particular of the canonical Lie bialgebroid morphism $E^*\to E$ defined for them (which is an abstract analog of the Poisson anchor, see ~\cite{mackenzie:book2005}).
We elaborate  these questions in~\cite{tv:microformal} and a forthcoming paper with H.~M.~Khudaverdian.

\appendix

\new{\section*{Appendix: ``nonlinear pullbacks'' for odd functions}}

In the main text we construct and study the mapping of \emph{even} functions on supermanifolds
\begin{equation*}
    \Phi^*\co \funn(M_2)\to \funn(M_1)\,
\end{equation*}
associated with a canonical relation $\Phi\subset T^*M_1\times (-T^*M_2)$.
There is a parallel construction of a similar mapping of \emph{odd} functions
\begin{equation*}
    \Psi^*\co \pfunn(M_2)\to \pfunn(M_1)\,.
\end{equation*}
Below we give a brief outline of the corresponding statements without repeating the proofs that generally go along the same lines. In the same way as the constructions in the main text are based on the symplectic geometry of the cotangent bundles of the (super)manifolds involved, the parallel constructions here make use of odd symplectic geometry. It is well known that there are fundamental differences between even and odd symplectic geometry (see, e.g.,~\cite{hov:deltabest,hov:khn1}, \cite{tv:laplace1}), but up to a certain point everything remains similar and it suffices for our purpose.

For a supermanifold $M$  consider the anticotangent bundle $\Pi T^*M$. If $x^a$ are  local coordinates on $M$, then on $\Pi T^*M$ we obtain local coordinates $x^a,x^*_a$, where the variables $x^*_a$ have the parities opposite to the parities of the corresponding $x^a$ and they transform as
\begin{equation*}
    x^*_a=\der{x^{a'}}{x^a}\,x^*_a\,.
\end{equation*}
The variables $x^a,x^*_a$ form canonically conjugate pairs w.r.t. the odd bracket (the canonical Schouten bracket), where
\begin{equation*}
    \lsch x^*_a,x^b \rsch =\d_a^b\,,
\end{equation*}
and $\lsch F,G\rsch = -(-1)^{(\Ft+1)(\Gt+1)}\lsch G,F\rsch\,$ (see, e.g.,~\cite{tv:graded}). It corresponds to the canonical odd symplectic form $\o=d(dx^a\,x^*_a)$.

\new{Let $\Psi\subset \Pi T^*M_1\times (-\Pi T^*M_2)$ be a canonical relation such that it can be specified by an odd generating function $\S=\S(x,y^*)$,\footnote{\new{Unlike the even case, a Lagrangian submanifold $\Lambda$ of an odd symplectic manifold $N$ has a discrete invariant. Namely, if $\dim N=n|n$, then $\dim \Lambda$ can take any of the values $n-k|k$, where $k=0,1,\ldots,n$. The relations $\Psi\subset \Pi T^*M_1\times (-\Pi T^*M_2)$ that we consider have this invariant equal to $m_1+n_2$, where $\dim M_1=n_1|m_1$ and $\dim M_2=n_2|m_2$.}}
\begin{equation*}
    \Psi =\left\{\,(x^a,x^*_a,y^i,y^*_i)\ \Bigl| \ x^*_a=\der{\S}{x^a}\,(x,y^*)\,, \ y^i=\der{\S}{y^*_i}\,(x,y^*)\,\Bigr.\right\}\,.
\end{equation*}
Here   $x^a,x^*_a$ are coordinates on $\Pi T^*M_1$ and $y^i,y^*_i$ are coordinates on $\Pi T^*M_2$. Then a given odd function $g\in\pfunn(M_2)$ is mapped to the odd function $f=:\Psi^*[g]\in \pfunn(M_1)$ defined by the formula
\begin{equation*}
    f(x)=g(y)+\S(x,y^*)-y^iy^*_i\,,
\end{equation*}
where
\begin{align*}
    y^*_i&=\der{g}{y^i}\,(y)\,, \intertext{and $y^i$ is determined from the equation}
    y^i&=\der{\S}{y^*_i}\,\Bigl(x,\der{g}{y}\,(y)\Bigr)\,
\end{align*}
(similarly to~\eqref{eq.f}, \eqref{eq.q}, \eqref{eq.y} above). This equation can be solved by iterations. If we expand
\begin{equation*}
    \S(x,y^*)=\S_0(x) + \f^i(x)y^*_i + \frac{1}{2}\,\S^{ij}(x)y^*_jy^*_i +\ldots \ ,
\end{equation*}
then the zeroth order term $\S_0$ is just a fixed odd function on $M_1$, the first order term corresponds to an ordinary smooth map $\f\co M_1\to M_2$, and the higher order terms give a `perturbation'. As in the main text, we obtain $\f_g\co M_1\to M_2$ as a perturbative series
\begin{equation*}
    \f_g=\f+\f_1[g]+\f_2[g]+ \ldots
\end{equation*}
with terms of orders $1, 2, \ldots$ in $g$.}

\smallskip
{\small

\begin{example}
\new{As an exercise, one can calculate the  linear and quadratic terms in $g$ to obtain
\begin{multline*}
    \f^i_g(x)=\f^i(x)\  + \ \underbrace{\S^{ij}(x)\der{g}{y^j}\,\bigl(\f(x)\bigr)}_{\f_1[g]}\  + \\
     \underbrace{\S^{i(j}(x)\S^{k)l}(x)\,\der{g}{y^l}\bigl(\f(x)\bigr)\,\dder{g}{y^k}{y^j}\bigl(\f(x)\bigr) +
    \frac{1}{2}\,\S^{ijk}(x)\,\der{g}{y^k}\,\bigl(\f(x)\bigr)\,\der{g}{y^j}\,\bigl(\f(x)\bigr)}_{\f_2[g]} \  + \ \ldots
\end{multline*}
(the round brackets in the indices denote symmetrization). The particular expression is not very important, but it gives a feeling of the general  appearance of the terms in the expansion.}
\end{example}

}

\new{Then the image of $\Psi^*$ in a greater detail is
\begin{equation*}
    \Psi^*[g](x)=g(\f_g(x))+\S\left(x,\der{g}{y}\,(\f_g(x))\right)-\f^i_g(x)\,\der{g}{y^i}\,(\f_g(x))\,,
\end{equation*}
and, as in the main text,  we can obtain that to the second order
\begin{equation*}
    \Psi^*[g](x)=\S_0(x)\ + \ g(\f(x))\  + \ \frac{1}{2}\,\S^{ij}(x)\,\der{g}{y^j}\,(\f(x))\,\der{g}{y^i}\,(\f(x))\ + \ \ldots
\end{equation*}
}

\new{Composition of canonical relations of the considered form   leads to another `formal category' extending  the category of smooth supermanifolds and their smooth maps, different from the one considered in the main text.\footnote{\new{Working in a formal framework  allows to go around the standard difficulties with composition. Compare remark at the end of section~\ref{sec.mainconstr}. Moreover, for ordinary (purely even) manifolds, the fibers of $\Pi T^*M$ are odd, hence there is no difference between formal and non-formal treatments. It was \v{S}evera~\cite{severa:oddsympl2004} who first noted that in that case Weinstein's symplectic ``category''   is a genuine category without quotes.}} One should  expect
\begin{equation*}
    (\Psi_1\circ \Psi_2)^*=\Psi_2^*\circ \Psi_1^*\,,
\end{equation*}
so ``nonlinear pullbacks'' give a nonlinear representation of this formal category on the spaces of odd functions.}

\new{Similarly to Theorem~\ref{thm.deriv} of the main text, we have}

\new{\begin{theore} The derivative of the formal nonlinear mapping
\begin{equation*}
    \Psi^*\co \pfunn(M_2)\to \pfunn(M_1)
\end{equation*}
at a point  $g\in \pfunn(M_2)$ is given by the formula:
\begin{equation*}
    (T\Psi^*)[g]=\f_g^*\,,
\end{equation*}
where
\begin{equation*}
    \f_g^*\co \fun(M_2)\to \fun(M_1)
\end{equation*}
is the  ordinary  pullback w.r.t. the   map 
$\f_g\co M_1\to M_2$ depending on $g$. \qed
\end{theore}}

\new{Analogs of the Hamilton--Jacobi vector fields introduced in the main text, in the `odd' setup take the form
\begin{equation*}
    \X_H = (-1)^{\Ht(m+1)}\int\limits_{M^{n|m}}\!  Dx\; H\left(\!x,\der{f}{x}\right)\var{}{f(x)}\,,
\end{equation*}
where $H\in\fun(\Pi T^*M)$ is a multivector (or `pseudomultivector') field on $M$.
Here we need to emphasize that the function $f$ is \emph{odd}, so in particular the substitution of its derivatives $\lder{f}{x^a}$ for the antimomenta $x^*_a$ makes good sense. In other words, we have infinitesimal shifts of odd functions on $M$ of the form
\begin{equation*}
    f\mapsto f_{\e}=f  +\e \X_H[f]\,, \ \text{where} \  \X_H[f](x)=  H\Bigl(x,\der{f}{x}\,(x)\Bigr)\,.
\end{equation*}
Here $\e^2=0$ and $\tilde\e=\Ht +1$. The parity of the vector field $\X_H$ on $\pfunn(M)$ is the opposite to the parity of $H$.}

\new{\begin{theore}
For arbitrary multivector fields $H$ and $F$,
\begin{equation*}
    \left[\X_H,\X_F\right]=(-1)^{\Ht}\X_{\lsch H,F\rsch}\,,
\end{equation*}
where the bracket at the l.h.s. is the commutator of vector fields on the infinite-dimensional supermanifold $\pfunn(M)$ and the bracket at the r.h.s. is the canonical Schouten bracket on $\Pi T^*M$. \qed
\end{theore}}

\new{Multivector fields $H_1\in\fun(\Pi T^*M_1)$ and $H_2\in\fun(\Pi T^*M_2)$ are said to be  {$R$-related} for  a relation $R\subset \Pi T^*M_1\times \Pi T^*M_2$ if $p_1^*H_1=p_2^*H_2$. For a canonical relation $\Psi\subset \Pi T^*M_1\times (-\Pi T^*M_2)$ as above the analog of Theorem~\ref{thm.main} holds:}

\new{\begin{theore} If multivector fields $H_1\in\fun(\Pi T^*M_1)$ and $H_2\in\fun(\Pi T^*M_2)$ are $\Psi$-related, then the   vector fields $\X_{H_2}\in\Vect(\pfunn(M_2))$  and $\X_{H_1}\in\Vect(\pfunn(M_1))$ are $\Psi^*$-related, for   $\Psi^*\co \pfunn(M_2)\to \pfunn(M_1)$\,. \qed
\end{theore}}

\new{An even multivector field $P\in\fun(\Pi T^*M)$ satisfying $\lsch P,P\rsch =0$
defines a \emph{homotopy Poisson structure} (or a \emph{$P_{\infty}$-structure}) on $M$ via the higher derived bracket construction~\cite{tv:higherder}.  That means     antisymmetric brackets of alternating parities  on $\fun(M)$ that make it into an $L_{\infty}$-algebra in the ``antisymmetric version'' and which are multiderivations w.r.t. ordinary multiplication. On the vector space $\Pi\!\fun(M)$ this induces an $L_{\infty}$-algebra structure in the ``symmetric version''.  With an abuse of language we still refer to $P$ as to a `Poisson tensor' on $M$. The homological vector field $\Q$ on the supermanifold $\pfunn(M)$ corresponding to this   $L_{\infty}$-structure has the Hamilton--Jacobi form
\begin{equation*}
    \Q = \int\limits_{M^{n|m}}\!  Dx\; P\Bigl(x,\der{f}{x}\,\Bigr)\var{}{f(x)}\,.
\end{equation*}}

\new{Let $M_1=(M_1,P_1)$ and $M_2=(M_2,P_2)$  be two homotopy Poisson  manifolds and let  $\Psi\subset \Pi T^*M_1\times (-\Pi T^*M_2)$ be a canonical relation as above.
\begin{coro}  If the Poisson tensors $P_1$ and $P_2$ are $\Psi$-related, then the mapping
\begin{equation*}
    \Psi^*\co \pfunn(M_2)\to \pfunn(M_1)
\end{equation*}
is an $L_{\infty}$-morphism of the corresponding $L_{\infty}$-algebras. \qed
\end{coro}}


\begin{thebibliography}{100}
\def\cprime{$'$}
\bibitem{arnold:mathmethodseng}
V.~I. Arnol{\cprime}d.
\newblock {\em Mathematical methods of classical mechanics}, volume~60 of {\em
  Graduate Texts in Mathematics}.
\newblock Springer-Verlag, New York, second edition, 1989.
\newblock Translated from the Russian by K. Vogtmann and A. Weinstein.

\bibitem{berezin:second}
F.~A. Berezin.
\newblock {\em The method of second quantization}.
\newblock Translated from the Russian by Nobumichi Mugibayashi and Alan
  Jeffrey. Pure and Applied Physics, Vol. 24. Academic Press, New York, 1966.

\bibitem{cattaneo-dherin-weinstein:one}
A.~S. Cattaneo, B. Dherin, and A. Weinstein.
\newblock Symplectic microgeometry {I}: {M}icromorphisms.
\newblock {\em J. Symplectic Geom.}, 8(2):205--223, 2010.

\bibitem{cattaneo-dherin-weinstein:two}
A.~S. Cattaneo, B. Dherin, and A. Weinstein.
\newblock Symplectic microgeometry {II}: {G}enerating functions.
\newblock {\em Bull. Braz. Math. Soc. (N.S.)}, 42(4):507--536, 2011.

\bibitem{cattaneo-dherin-weinstein:three}
A.~S. Cattaneo, B. Dherin, and A. Weinstein.
\newblock Symplectic microgeometry {III}: {M}onoids.
\newblock {\em J. Symplectic Geom.}, 11(3):319--341, 2013.

\bibitem{cattaneo-dherin-weinstein:comorphisms}
A.~S. Cattaneo, B. Dherin, and A. Weinstein.
\newblock Integration of {L}ie algebroid comorphisms.
\newblock {\em Port. Math.}, 70(2):113--144, 2013.
\newblock {\tt arXiv:1210.4443 [math.DG]}.

\bibitem{hov:deltabest}
H[O].~M. Khudaverdian.
\newblock Geometry of superspace with even and odd brackets.
\newblock Preprint of the {Geneva University}, {UGVA-DPT} 1989/05-613, 1989.
  Published in: \textit{J. Math. Phys.} 32 (1991), 1934--1937.

\bibitem{hov:khn1}
H[O].~M. Khudaverdian and A.~P. Nersessian.
\newblock On geometry of {Batalin-Vilkovisky} formalism.
\newblock {\em Mod. Phys. Lett}, A8(25):2377--2385, 1993.


\bibitem{tv:laplace1}
H.~M.~Khudaverdian and Th.~Th.~ Voronov.
\newblock On odd {Laplace} operators.
\newblock {\em Lett. Math. Phys.}, 62:127--142, 2002.

\bibitem{tv:higherpoisson}
H.~M. Khudaverdian and Th.~Th. Voronov.
\newblock Higher {Poisson} brackets and differential forms.
\newblock In {\em X{XVII} {W}orkshop on {G}eometrical {M}ethods in {P}hysics},
  volume 1079 of {\em AIP Conf. Proc.}, pages 203--215. Amer. Inst. Phys.,
  Melville, NY, 2008.


\bibitem{kontsevich:quant}
M.~Kontsevich.
\newblock Deformation quantization of {Poisson} manifolds, {I}.
\newblock {\tt math.QA/9709180}.

\bibitem{kontsevich:quantlmp}
M.~Kontsevich.
\newblock Deformation quantization of {P}oisson manifolds.
\newblock {\em Lett. Math. Phys.}, 66(3):157--216, 2003.

\bibitem{lada:stasheff}
T. Lada and J. Stasheff.
\newblock Introduction to {SH} {L}ie algebras for physicists.
\newblock {\em Internat. J. Theoret. Phys.}, 32(7):1087--1103, 1993.

\bibitem{maclane:homology}
S.~Mac~Lane.
\newblock {\em Homology}.
\newblock Die Grundlehren der mathematischen Wissenschaften, Bd. 114. Academic
  Press Inc., Publishers, New York, 1963.

\bibitem{mackenzie:book2005}
K.~C.~H. Mackenzie.
\newblock {\em General theory of {L}ie groupoids and {L}ie algebroids}, volume
  213 of {\em London Mathematical Society Lecture Note Series}.
\newblock Cambridge University Press, Cambridge, 2005.

\bibitem{mackenzie:bialg}
K.~C.~H. Mackenzie and P.~ Xu.
\newblock Lie bialgebroids and {P}oisson groupoids.
\newblock {\em Duke Math. J.}, 73(2):415--452, 1994.

\bibitem{neretin:categories}
Yu.~A. Neretin.
\newblock Extension of representations of classical groups to representations
  of categories.
\newblock {\em Algebra i Analiz}, 3(1):176--202, 1991.

\bibitem{neretin:categories-book}
Yu.~A. Neretin.
\newblock {\em Categories of symmetries and infinite-dimensional groups},
  volume~16 of {\em London Mathematical Society Monographs. New Series}.
\newblock The Clarendon Press, Oxford University Press, New York, 1996.
\newblock Translated from the Russian by G. G. Gould, Oxford Science
  Publications.

\bibitem{schwarz:bv}
A.~Schwarz.
\newblock Geometry of {B}atalin-{V}ilkovisky quantization.
\newblock {\em Comm. Math. Phys.}, 155(2):249--260, 1993.

\bibitem{severa:oddsympl2004}
\new{P.~{\v{S}}evera.
\newblock Noncommutative differential forms and quantization of the odd
  symplectic category.
\newblock {\em Lett. Math. Phys.}, 68(1):31--39, 2004.}

\bibitem{tv:graded}
Th.~Th.~Voronov.
\newblock Graded manifolds and {Drinfeld} doubles for {Lie} bialgebroids.
\newblock In  {\em Quantization, Poisson Brackets and
  Beyond}, volume 315 of {\em Contemp. Math.}, pages 131--168. Amer. Math.
  Soc., Providence, RI, 2002.

\bibitem{tv:higherder}
Th.~Th.~Voronov.
\newblock Higher derived brackets and homotopy algebras.
\newblock {\em J. of Pure and Appl. Algebra}, 202(1--3):133--153, 2005.

\bibitem{tv:higherderarb}
Th.~Th. Voronov.
\newblock Higher derived brackets for arbitrary derivations.
\newblock In {\em Travaux math\'ematiques. {F}asc. {XVI}}, Trav. Math., XVI,
  pages 163--186. Univ. Luxemb., Luxembourg, 2005.

\bibitem{tv:microformal}
Th.~Th. Voronov.
\newblock Microformal geometry.
\newblock \texttt{arXiv:1411.6720 [math.DG]}.

\bibitem{tv:oscil}
Th.~Th. Voronov.
Thick morphisms of supermanifolds and oscillatory integral operators.
\newblock {\em Russian Math. Surveys}, 71 (6), 2016 \emph{\mbox{(in print)}}.
\newblock \href{http://arxiv.org/abs/1506.02417}{\texttt{arXiv:1506.02417 [math.DG]}}.

\bibitem{tv:qumicro}
Th.~Th. Voronov.
Quantum microformal  morphisms of supermanifolds: an explicit formula  and further properties.
\newblock \href{http://arxiv.org/abs/1512.04163}{\texttt{arXiv:1512.04163 [math-ph]}}.


\bibitem{weinstein:sympl71}
A.~Weinstein.
\newblock Symplectic manifolds and their {L}agrangian submanifolds.
\newblock {\em Advances in Math.}, 6:329--346 (1971), 1971.

\bibitem{weinstein:symplcat82}
\new{A.~Weinstein.
\newblock The symplectic ``category''.
\newblock In {\em Differential geometric methods in mathematical physics
  ({C}lausthal, 1980)}, volume 905 of {\em Lecture Notes in Math.}, pages
  45--51. Springer, Berlin-New York, 1982.}

\bibitem{weinstein:symplectic-geometry1981}
A.~Weinstein.
\newblock Symplectic geometry.
\newblock {\em Bull. Amer. Math. Soc. (N.S.)}, 5(1):1--13, 1981.

\bibitem{weinstein:symplectic-geometries}
A.~Weinstein.
\newblock Symplectic categories.
\newblock {\em Port. Math.}, 67(2):261--278, 2010.
\end{thebibliography}

\end{document}